\documentclass[11pt,a4paper]{article}

\usepackage{caption2}
\usepackage{CJK}
\usepackage{tabls}
\usepackage{bm}
\usepackage{multirow}
\usepackage{amssymb}
\usepackage{amsfonts}
\usepackage{mathrsfs}
\usepackage{empheq}
\usepackage{amsmath}
\usepackage{amsthm}
\usepackage{graphicx}
\usepackage{color}
\usepackage{indentfirst}
\usepackage{hyperref}
\usepackage{float}
\usepackage{cases}
\usepackage[titletoc]{appendix}
\usepackage{bm}

\usepackage[colorinlistoftodos,english]{todonotes}

\allowdisplaybreaks%

\def\pa{\partial}

\def\i1n{i=1,\cdots,n}
\def\j1n{j=1,\cdots,n}
\def\ij1n{i,j=1,\cdots,n}

\def \i{\mathrm i}

 \numberwithin{equation}{section}
\theoremstyle{definition}

 \newtheorem{theorem}{Theorem}[section]
 \newtheorem{lemma}[theorem]{Lemma}
 \newtheorem{corollary}[theorem]{Corollary}

 \newtheorem{remark}[theorem]{Remark}

\theoremstyle{definition}
\theoremstyle{theorem}
\theoremstyle{lemma}

\newcommand{\be}{\begin{equation}}
\newcommand{\ee}{\end{equation}}
\newcommand{\beq}{\begin{equation*}}
\newcommand{\eeq}{\end{equation*}}

\begin{document}

\title{Physical space approach to wave equation bilinear estimates revisit}


\author{Sheng Wang\thanks{ Shanghai Center for Mathematical Sciences, Fudan University, Shanghai 200433, P.R. China (19110840011@fudan.edu.cn).}
\and Yi Zhou\thanks{School of Mathematical Sciences, Fudan University, Shanghai 200433, P.R. China (yizhou@fudan.edu.cn).}}

\maketitle

\begin{abstract}
	In the paper by Klainerman, Rodnianski and Tao \cite{Klainerman-Rodnianski-Tao}, they give a physical space proof to a classical result of Klainerman and Machedon \cite{Klainerman-Machedon} for the bilinear space-time estimates of null forms. In this paper, we shall give an alternative and very simple physical space proof of the 
	same bilinear estimates by applying div-curl type lemma of Zhou \cite{Zhou} and Wang and Zhou \cite{Wang-Zhou-1}, \cite{Wang-Zhou-2}. As far as we known, the later development of wave maps \cite{Sterbenz-1}, \cite{Sterbenz-2}, \cite{Tao-1}, \cite{Tao-2}, \cite{Tataru-1}, \cite{Tataru-2}, and the proof of bounded  curvature theorem \cite{Klainerman-Rodnianski-Szeftel-1}, \cite{Klainerman-Rodnianski-Szeftel-2}  rely on basic ideas of Klainerman and Machedon \cite{Klainerman-Machedon} as well as Klainerman, Rodnianski and Tao \cite{Klainerman-Rodnianski-Tao}.

\end{abstract}

\textbf{ 2010 Mathematics Subject  Classification.}


\section{Introduction}\label{intro}
In a classical paper by Klainerman and Machedon \cite{Klainerman-Machedon}, they proved the following estimate for the null forms by purely Fourier transformation. 

\begin{align}\label{bies}
	&\|Q\left(\phi, \psi\right)\|_{L^2(\mathbb{R}^{+}\times\mathbb{R}^n)}\nonumber\\
& \lesssim \left(\|\phi_0\|_{\dot{H}^{\frac{n+1}{2}}\left(\mathbb{R}^n\right)}+\|\phi_1\|_{\dot{H}^{\frac{n-1}{2}}\left(\mathbb{R}^n\right)}\right)\cdot \left(\|\psi_0\|_{\dot{H}^{1}\left(\mathbb{R}^n\right)}+\|\psi_1\|_{L^2\left(\mathbb{R}^n\right)}\right)
\end{align}
where $Q$ is any of the following null form:
\begin{equation}
\label{main}
\left\{
\begin{aligned}
& Q_0\left(\phi, \psi\right)=\pa_t\phi\pa_t \psi- \sum\limits_{i}^{n} \pa_{x_i}\phi\pa_{x_i} \psi, \\
& Q_{\alpha\beta}\left(\phi,\psi\right)=\pa_\alpha\phi\pa_\beta\psi-\pa_\alpha\psi\pa_\beta\phi,\\
\end{aligned}
\right.
\end{equation}
 $\alpha, \beta
  = 0  , 1,  \dots n$. Here, we denote $\pa_0 := \pa_t$, $\pa_i:=\pa_{x_i},  i= 1,2 \dots n$.  And $\phi$, $\psi$ satisfies the   D' Alembertian equation:
\begin{equation}
\label{main1}
\left\{
\begin{aligned}
&\Box \phi\left(t,x\right)=0, \ t>0, \ x\in\mathbb{R}^n \\
&t=0: \ \phi=\phi_0\left(x\right), \pa_t\phi =\phi_1\left(x\right), \ x\in\mathbb{R}^n  \\
\end{aligned}
\right.
\end{equation}

\begin{equation}
\label{main2}
\left\{
\begin{aligned}
&\Box \psi\left(t,x\right)=0, \ t>0, \ x\in\mathbb{R}^n \\
&t=0: \ \psi=\psi_0\left(x\right), \pa_t\psi =\psi_1\left(x\right), \ x\in\mathbb{R}^n  \\
\end{aligned}
\right.
\end{equation}
where $\Box := \pa^2_t - \sum\limits^{n}_{i=1} \pa^2_{x_i}$.

Then in another paper by Klainerman, Rodnianski, and Tao \cite{Klainerman-Rodnianski-Tao}, they give a physical space approach to those estimates and their generalizations. The aim of this paper is to give an alternative and very simple physical space proof of  the classical wave equation bilinear estimates of Klainerman and Machedon \cite{Klainerman-Machedon} by using div-curl type lemma of Zhou \cite{Zhou} and Wang and Zhou \cite{Wang-Zhou-1}, \cite{Wang-Zhou-2}. We have only attained the limited goal of proving the bilinear estimates for the dyadic piece of the solution. Summing up the dyadic parts leads to  the bilinear estimates with a Besov loss. As far as we known, the later development of wave maps \cite{Sterbenz-1}, \cite{Sterbenz-2}, \cite{Tao-1}, \cite{Tao-2}, \cite{Tataru-1}, \cite{Tataru-2}, and the proof of bounded  curvature theorem \cite{Klainerman-Rodnianski-Szeftel-1}, \cite{Klainerman-Rodnianski-Szeftel-2} rely on basic ideas of those two papers.

Firstly, we introduce some notations to facilitate writing. 

We mark $ A \lesssim B $ to mean there exists a constant $ C > 0 $ such that $ A \leqslant C B $. 

Let $u$ satisfies D' Alembertian equation, 
\begin{equation}\label{Bu}
	\Box u =0.
\end{equation}

We define the energy density for  D' Alembertian equation:
\begin{align}
e(u):= \frac{1}{2} \left(\lvert \pa_t u \rvert^2 + \lvert \nabla u\rvert^2\right),
\end{align}
and the energy for  D' Alembertian equation: 
\begin{align}
E(u):=   \int_{\mathbb{R}^{n}} e(u) dx.
\end{align}	

In fact, if $u$ is the soulution to the Cauchy problem of D' Alembertian equation, then we know: 
\begin{align}\label{nlsh}
\frac{d E\left(u \left(t\right)\right)}{dt}=0.
\end{align}

Next, we define:
\begin{align*}
&\phi:= \sum\limits_{\mu } \phi_\mu, \\
&\psi:= \sum\limits_{\lambda } \psi_\lambda,
\end{align*}
where $\lambda$ and $\mu$ are all dyadic numbers,  be the standard homogeneous Littlewood-Paley decomposition.

Now, we give  an alternative physical space proof of the following local in time version of  bilinear estimates of Klainerman, Rodnianski and Tao \cite{Klainerman-Rodnianski-Tao} and  Klainerman, Machedon \cite{Klainerman-Machedon}:
\begin{theorem}\label{wbes}
	\begin{align}
	\|Q\left(\phi_\mu, \psi_\lambda\right)\|_{L^2([0,1]\times\mathbb{R}^n)}\lesssim \mu^{\frac{n-1}{2}}E^{1/2}\left(\phi_{\mu}\right)E^{1/2}\left(\psi_{\lambda}\right)
	\end{align}
	for $\mu \leq \lambda$ and for any of null forms $Q_0$  and $Q_{\alpha, \beta}$.
\end{theorem}

\begin{theorem}\label{dwbes}
	\begin{align}
	\|\left(-\Delta\right)^{-1/2}Q_{i,j}\left(\phi_\mu, \psi_\lambda\right)\|_{L^2([0,1]\times\mathbb{R}^n)}\lesssim \mu^{\frac{n-3}{2}}E^{1/2}\left(\phi_{\mu}\right)E^{1/2}\left(\psi_{\lambda}\right)
	\end{align}
	for $\mu \leq \lambda$, $n\geq 3$ and for   $1 \leq i , j \leq n$.
\end{theorem}

\section{Div-Curl Lemma}
\begin{lemma}\label{dcr}
	Suppose that
	\begin{equation}
	\left\{
	\begin{aligned}
	&\partial_t f^{11} + \partial_x  f^{12} =G^1,\\
	& \partial_t f^{21}-\partial_x f^{22}=G^2,
	\end{aligned}
	\right.
	\end{equation}
	
	\begin{equation}
	\begin{aligned}
	f^ {11}, f^{12}, f^{21}, f^{22} \rightarrow 0, x\rightarrow \infty,
	\end{aligned}
	\end{equation}
where $\left(t,x\right) \in \mathbb{R}_{+}\times \mathbb{R}^1 $ and each $f^{i,j} \left(i,j = 1,2\right)$  is a   real-value function of $\left(t,x\right)$ .
	
	Then  we have:
	
	\begin{align}
	\int_{0}^{T}\int_{-\infty}^{+\infty} f^{11}f^{22}+f^{12}f^{21} &\lesssim
	\left(\|f^{11}\left(0\right)\|_{L^1} + \sup\limits_{0\leq t\leq T}  \|f^{11}\left(t\right)\|_{L^1} + \int_{0}^{T}\int_{-\infty}^{+\infty}\lvert G^1\rvert \right)\nonumber\\
	&\cdot\left(\|f^{21}\left(0\right)\|_{L^1} +\sup\limits_{0\leq t\leq T} \|f^{21}\left(t\right)\|_{L^1}  +\int_{0}^{T}\int_{-\infty}^{+\infty}\lvert G^2\rvert \right)\nonumber
	\end{align}
	provided that the right side is bounded.
\end{lemma}

\begin{proof} 
	The details of the proof can be referred to in our previous work \cite{Wang-Zhou-2}, and are repeated here for the convenience of the reader.

	\begin{equation*}
	\partial_t \left(\int_{-\infty}^{x} f^{11} \right) +f^{12}=\int_{-\infty}^{x} G^1.
	\end{equation*}
	
	Then 
	\begin{equation}\label{dd1r}
	f^{21}\int_{-\infty}^{x} \partial_t f^{11}+f^{12}f^{21}=f^{21}\int_{-\infty}^{x}G^1,
	\end{equation}
	
	\begin{equation}\label{dd2r}
	\partial_t f^{21}\int_{-\infty}^{x} f^{11}-\partial_x f^{22}\int_{-\infty}^{x}f^{11}= G^2\int_{-\infty}^{x}f^{11}.
	\end{equation}
	
	\eqref{dd1r}+\eqref{dd2r}:
	\begin{align}
	\int_{-\infty}^{+\infty} \partial_t \left(\int_{-\infty}^{x} f^{11}f^{21}\right)+&\int_{-\infty}^{+\infty}
	\left(f^{12}f^{21}-\partial_xf^{22}\int_{-\infty}^{x}f^{11}\right)\nonumber\\ 
	&=\int_{-\infty}^{+\infty}\left(f^{21}\int_{-\infty}^{x}G^1 + G^2\int_{-\infty}^{x}f^{11}\right)\nonumber.
	\end{align}
	
	We have:
	
	\begin{align}
	\int_{0}^{T}\int_{-\infty}^{+\infty} f^{11}f^{22}+f^{12}f^{21}=& \int_{-\infty}^{+\infty} \left(\int_{-\infty}^{x} f^{11}f^{21}\right)\left(0\right)- \left(\int_{-\infty}^{x} f^{11}f^{21}\right)\left(T\right) \nonumber\\
	&+\int_{0}^{T}\int_{-\infty}^{+\infty}\left(f^{21}\int_{-\infty}^{x}G^1 + G^2\int_{-\infty}^{x}f^{11}\right)\nonumber\\
	&:= \mathcal{A}_1+\mathcal{A}_2+\mathcal{A}_3,
	\end{align}
	
	where	
	
	\begin{align}
	\lvert \mathcal{A}_1\rvert&\lesssim \|f^{11}\left(0\right)\|_{L^1}\|f^{21}\left(0\right)\|_{L^1} + \|f^{11}\left(T\right)\|_{L^1}\|f^{21}\left(T\right)\|_{L^1},\nonumber
	\end{align}

	\begin{align}
	\lvert \mathcal{A}_2\rvert&\lesssim\int_{0}^{T} \|f^{21}\left(t\right)\|_{L^1}\|G^1\left(t\right)\|_{L^1}\nonumber\\
	\lesssim&\sup\limits_{0\leq t\leq T} \| f^{21}\left(t\right)\|_{L^1}\left(\int_{0}^{T}\int_{-\infty}^{+\infty}\lvert G^1\rvert\right),\nonumber
	\end{align}

	\begin{align}
	\lvert \mathcal{A}_3\rvert&\lesssim\int_{0}^{T} \|f^{11}\left(t\right)\|_{L^1}\|G^2\left(t\right)\|_{L^1}\nonumber\\
	\lesssim&\sup\limits_{0\leq t\leq T} \| f^{11}\left(t\right)\|_{L^1}\left(\int_{0}^{T}\int_{-\infty}^{+\infty}\lvert G^2\rvert\right)\nonumber.
	\end{align}

	Based analysis above, we complete the proof. 
\end{proof}

\begin{remark}\label{neiji}
	In fact, the quantity $f^{11}f^{22}+f^{12}f^{21}$ is an ``inner product" in some sense. But by coincidence, this quantity equals to the determinant of a matrix.  We  can represent the matrix in the following form
	
	\begin{equation*}
	\mathbb{A}=
	\begin{pmatrix}
	f^{11} & -f^{12} \\
	f^{21}           & f^{22}
	\end{pmatrix}.
	\end{equation*}
	
\end{remark}

\section{Application of div-curl lemma}
In this section, we assume the functions $u$ and $v$ satisfy the D' Alembertian equation respectively, and obatin the energy conservation law and ``momentum" balence law. Combining this and  applying the div-crul lemma,  we  derive some mixed type estimates. Here, we denote $x:=\left(x_1, y\right)$,  where $y=\left(x_2,\dots, x_n\right)$.

Let $u$ satisfies
\begin{equation}\label{bu}
\Box u =0.
\end{equation}

Multiply \eqref{bu} by $\pa_t u $, we get the energy conservation law:
\begin{align}
\pa_t e\left(u\right) -\sum\limits_{i=1}^{n} \pa_{x_i} \left(\pa_t u \pa_{x_i}u\right) =0.
\end{align}
Integrate in $y$, we get:
\begin{align}\label{bc1}
\pa_t \int_{\mathbb{R}^{n-1}} e\left(u\right) dy - \pa_{x_1} \int_{\mathbb{R}^{n-1}}\left(\pa_tu\pa_{x_1}u\right)dy
=0.
\end{align}

Let $v$ satisfies
\begin{equation}\label{bv}
\Box v =0
\end{equation}
we rewrite  \eqref{bv} as 
\begin{equation} \label{rdh}
	\pa^2_t v-\pa^2_{x_1}v+\Delta_y v=2\Delta_y v.
\end{equation}

Multiply \eqref{rdh} by $\pa_{x_1}v$, and integrate in $y$, we get:
\begin{align}\label{bc2}
\pa_t \int_{\mathbb{R}^{n-1}} \left(\pa_tv\pa_{x_1}v\right) dy - \pa_{x_1} \int_{\mathbb{R}^{n-1}} e\left(v\right)dy= 2\int_{\mathbb{R}^{n-1}} \pa_{x_1}v \Delta_yvdy.
\end{align}

\begin{theorem}\label{ges}
	For solution $u$ and $v$ to the D' Alembertian equation with initial data in Schwartz class $\mathcal{S}(\mathbb{R}^n)$,  the we obtain a general estimate of the form:
	\begin{align}
		&\int_{0}^{1} \int_{\mathbb{R}} \frac{1}{4} E_1(u)E_2(v) + \frac{1}{4} E_1(v)E_2(u) + \frac{1}{8} D^+(u)D^-(v)+\frac{1}{8}D^-(u)D^+(v)dx_1dt\nonumber\\
		&\lesssim E(u)E(v) + E(u)\|\Delta_y v\|^2_{L^2\left([0,1]\times \mathbb{R}^n\right)},
	\end{align}
\end{theorem}
where some notations are defined as follows:
\begin{align*}
&D^{\pm}(f):=  \int_{\mathbb{R}^{n-1}} \left(\pa_t f \pm \pa_{x_1} f\right)^2dy, \quad P(f):= \int_{\mathbb{R}^{n-1}} \pa_tf \pa_{x_1}fdy\nonumber\\
&E_1(f):= \int_{\mathbb{R}^{n-1}} e(f)dy, \quad\ E_2(f):=\int_{\mathbb{R}^{n-1}} |\nabla_yf|^2 dy.
\end{align*}
\begin{proof}
Firstly, we use the div-curl lemma to the balence law  \eqref{bc1} and \eqref{bc2}
to obtain:
\begin{align}
& \int_{0}^{1}\int_{-\infty}^{+\infty} \left(\int_{\mathbb{R}^{n-1}} e\left(u\right)dy\int_{\mathbb{R}^{n-1}} e\left(v\right)dy\right)dx_1dt\nonumber\\
&-\int_{0}^{1}\int_{-\infty}^{+\infty}\left(\int_{\mathbb{R}^{n-1}}\pa_tu\pa_{x_1}udy\int_{\mathbb{R}^{n-1}} \pa_tv\pa_{x_1}vdy\right)d x_1 dt\nonumber\\
&\lesssim \left( \sup\limits_{t}\int_\mathbb{R} \lvert \int_{\mathbb{R}^{n-1}} e\left(u\right) dy\rvert dx_1\right)\cdot \left( \sup\limits_{t}\int_\mathbb{R} \lvert \int_{\mathbb{R}^{n-1}} \pa_tv\pa_{x_1}v dy\rvert dx_1\right)\nonumber\\
&+ \left( \sup\limits_{t}\int_\mathbb{R} \lvert \int_{\mathbb{R}^{n-1}} e\left(u\right) dy\rvert dx_1\right)\cdot\left(\int_{0}^{1} \int_\mathbb{R} \lvert \int_{\mathbb{R}^{n-1}} \pa_{x_1} v \Delta_y vdy\rvert dx_1dt\right)\nonumber\\
&\lesssim\left(\sup\limits_t \int_{\mathbb{R}^{n}} e\left(u\right)dx\right)\cdot\left(\sup\limits_t \int_{\mathbb{R}^{n}} \rvert \pa_tv\pa_{x_1}v \lvert dx\right)\nonumber\\
&+\left(\sup\limits_t \int_{\mathbb{R}^{n}} e\left(u\right)dx\right)\cdot\left(\int_{0}^{1}\int_{\mathbb{R}^{n}} \lvert  \Delta_y v\rvert^2 dxdt\right)\nonumber\\
&+\left(\sup\limits_t \int_{\mathbb{R}^{n}} e\left(u\right)dx\right)\cdot\left(\int_{0}^{1}\int_{\mathbb{R}^{n}} \lvert  \nabla v\rvert^2 dxdt\right)
\end{align}

It is obvious that
\begin{align*}
&e\left(f\right)=\frac{1}{4}\left(\left(\partial_t f + \partial_{x_1}f\right)^2+\left(\partial_t f - \partial_{x_1}f\right)^2\right)+\frac{1}{2}\lvert \nabla_y f \rvert^2,\\
&\partial_tf\partial_{x_1}f= \frac{1}{4}\left(\left(\partial_t f + \partial_{x_1}f\right)^2-\left(\partial_t f - \partial_{x_1}f\right)^2\right).\nonumber
\end{align*}

Based above, by some simple calculations, we have:
\begin{align}
&\int_{\mathbb{R}^{n-1}} e\left(u\right)\int_{\mathbb{R}^{n-1}} e\left(v\right)-\int_{\mathbb{R}^{n-1}}\pa_t u\pa_{x_1} u\int_{\mathbb{R}^{n-1}} \pa_t v \pa_{x_1} v\nonumber\\
&=E_1(u)E_1(v)-P(u)P(v)\nonumber\\
&=\frac{1}{16}\left( D^+(u)+D^-(u)\right)\left( D^+(v)+D^-(v)\right)+\frac{1}{8}\left( D^+(u)+D^-(u)\right) E_2(v)\nonumber\\
&+\frac{1}{8}\left(D^+(v)+D^-(v)\right)E_2(u)-\frac{1}{16}\left( D^+(u)-D^-(u)\right)\left( D^+(v)-D^-(v)\right)\nonumber\\
&+ \frac{1}{4} E_2(u)E_2(v)\nonumber\\
&=\frac{1}{8} \left( D^+(u)+D^-(u)\right) E_2(v)+\frac{1}{8} \left( D^+(v)+D^-(v)\right) E_2(u)\nonumber\\
&+\frac{1}{4} E_2(u)E_2(v)\nonumber\\
&+\frac{1}{8}D^+(u)D^-(v)+\frac{1}{8}D^-(v)D^+(u)\nonumber\\
&=\frac{1}{4} E_1(u)E_2(v)+\frac{1}{4} E_1(v)E_2(u)\nonumber\\
&+\frac{1}{8}D^+(u)D^-(v)+\frac{1}{8}D^-(u)D^+(v)\nonumber\\
&+\frac{1}{16} \left( D^+(u)+D^-(u)\right) E_2(v)+\frac{1}{16} \left( D^+(v)+D^-(v)\right) E_2(u)\nonumber
\end{align}

Combining above and noting \eqref{nlsh}, we finish our proof.
\end{proof}

\section{Proof of bilinear null form estimates}
We first prove Theorem	\ref{wbes}.
\begin{proof}
We first consider the case   for $Q=Q_0$. 
Let 
\begin{align}\label{fj}
\phi_{\mu} = \sum\limits_{i=1}^{c\mu^{\frac{n-1}{2}}} \phi_{\mu, \omega_i},
\end{align}
where $\phi_{\mu, \omega_i}$ are real functions.  Moreover, $\hat{\phi}_{\mu, \omega_i}\left(\xi\right):= Q^{\mu}_{\omega_i}\left(\frac{\xi}{|\xi|}\right) \hat{\phi}_\mu\left(\xi\right)$, where  $Q^{\mu}_{\omega_i}\left(\frac{\xi}{|\xi|}\right)$ is the  partition of unity on sphere statisfying $\sum\limits^{c\mu^{\frac{n-1}{2}}}_{i=1}Q^{\mu}_{\omega_i}= 1$, and in the support of $Q^{\mu}_{\omega_i}$ 
\begin{align*}
\sum\limits^{n}_{j,k =1} |\omega_i^{(k)}\xi_j -\omega_i^{(j)}\xi_k| \lesssim  \mu^{-1/2}\lvert \xi \rvert .
\end{align*}
So in the support of $\hat{\phi}_{\mu, \omega_i}\left(\xi\right)$, we have
\begin{align}\label{wpd}
	\sum\limits^{n}_{j,k =1} |\omega_i^{(k)}\xi_j -\omega_i^{(j)}\xi_k| \lesssim  \lvert \xi \rvert^{1/2} .
\end{align}

Here, we use the notations $\omega_i^{(l)}$ to denote the $l$ component of the vector $\omega_i:=(\omega_i^{(1)}, \dots, \omega_i^{(l)}, \dots \omega_i^{(n)})$ and  $\xi_l$  to denote the $l$ component of the vector $\xi: =(\xi_1, \dots, \xi_l, \dots \xi_n)$.
So $\hat{\phi}_{\mu, \omega_i}$ is concentrated both in $\omega_i$ and $-\omega_i$ directions and is a real function. In the paper of Klainerman, Rodnianski and Tao \cite{Klainerman-Rodnianski-Tao}, they also considered the decomposition of solution in difference direction $\omega_i$ 's.
Thus, from the decomposition \eqref{fj}, we have: \begin{align}
\|Q_0\left(\phi_\mu, \psi_\lambda\right)\|_{L^2([0,1]\times\mathbb{R}^n)}\lesssim
\sum\limits_{i=1 }^{c\mu^{\frac{n-1}{2}}}  \|Q_0\left(\phi_{\mu, \omega_i}, \psi_\lambda\right)\|_{L^2([0,1]\times\mathbb{R}^n)}.
\end{align}    
It is enough to prove:
\begin{align}\label{fenes}
	I_{\omega_i}:= 	\|Q_0\left(\phi_{\mu,\omega_i}, \psi_\lambda\right)\|_{L^2([0,1]\times\mathbb{R}^n)} \lesssim \mu^{\frac{n-1}{4}}E^{1/2}\left(\phi_{\mu,\omega_i}\right)E^{1/2}\left(\psi_{\lambda}\right).
\end{align}

Because by \eqref{fenes}, we can get:
\begin{align}
	&\|Q_0\left(\phi_\mu, \psi_\lambda\right)\|_{L^2([0,1]\times\mathbb{R}^n)}\nonumber\\
	&\lesssim\mu^{\frac{n-1}{4}}\left(\sum\limits_{i=1 }^{c\mu^{\frac{n-1}{2}}} 1\right)^{1/2}\left(\sum\limits_{i=1} E\left(\phi_{\mu, \omega_i}\right)\right)^{1/2}E^{1/2}\left(\psi_{\lambda}\right)\nonumber\\
	&\lesssim \mu^{\frac{n-1}{2}}E^{1/2}\left(\phi_{\mu}\right)E^{1/2}\left(\psi_{\lambda}\right).
\end{align}

To prove \eqref{fenes}, by rotational invarience, without loss of generality, we can assume $\omega_i=e_1$. 

It is not difficult to know $\phi_{\mu}$, $\psi_{\lambda}$, $\phi_{\mu,e_1}$ all  statify the D' Alembertian equation. So  we apply the theorem \ref{ges}  to  $\psi_{\lambda} $ and $\phi_{\mu,e_1}$,  and noting \eqref{nlsh}, we have:
\begin{align}\label{hxg}
	&\int_{0}^{1} \int_{\mathbb{R}} \frac{1}{4} E_1(\phi_{\mu,e_1})E_2(\psi_{\lambda}) + \frac{1}{4} E_1(\psi_{\lambda})E_2(\phi_{\mu,e_1}) + \frac{1}{8} D^+(\phi_{\mu,e_1})D^-(\psi_{\lambda})+\frac{1}{8}D^-(\phi_{\mu,e_1})D^+(\psi_{\lambda})dx_1dt\nonumber\\
&\lesssim E(\phi_{\mu,e_1})E(\psi_{\lambda}) +  E(\psi_{\lambda})\|\Delta_y \phi_{\mu,e_1}\|^2_{L^2\left([0,1]\times \mathbb{R}^n\right)},\nonumber\\
&\lesssim E(\phi_{\mu,e_1})E(\psi_{\lambda})+ E(\psi_{\lambda})\|\nabla \phi_{\mu,e_1}\|^2_{L^2\left([0,1]\times \mathbb{R}^n\right)},\nonumber\\
&\lesssim E(\phi_{\mu,e_1})E(\psi_{\lambda})
\end{align}

Here, we use \eqref{wpd}, so two derivative in $y$ is like one derivatives in $x_1$ and $\|\Delta_y \phi_{\mu,e_1}\|_{L^2} \lesssim \|  \nabla \phi_{\mu, e_1}\|_{L^2}$. More precisely, 
 denoting $\xi=(\xi_1, \xi^{'})$ and taking $\omega_i=e_1$ in \eqref{wpd}, we obatin:
\begin{align}\label{y1x}
	\int \lvert \Delta_y \phi_{\mu, e_1} \rvert^2 dx =& \int |\xi ^ {'}|^4 \lvert \hat{\phi}_{\mu, e_1} \rvert^2 d\xi \nonumber\\
	&\lesssim \int |\xi|^2  \lvert \hat{\phi}_{\mu, e_1} \rvert^2 d\xi \sim \int |\nabla \phi_{\mu, e_1}|^2 d\xi,
\end{align}
this is the reason why we put the term $2 \Delta_y \phi_{\mu, e_1}$ on the right hand side as an inhomogeous term.

So we have:
\begin{equation}
	I_{e_1}\leq \left(\int_{0}^{1}\int_{-\infty}^{+\infty} h^2_{e_1}\left(t,x_1\right)dtdx_1\right)^{1/2},
\end{equation}
where
\begin{align}
h_{e_1}&=\|\pa_t \psi_{\lambda}-\pa_{x_1}\psi_{\lambda}\|_{L^2_y}\|\pa_t\phi_{\mu, e_1}+\pa_{x_1}\phi_{\mu, e_1}\|_{L^{\infty}_y}\nonumber\\
&+\|\pa_t\psi_{\lambda}+\pa_{x_1}\psi_{\lambda}\|_{L^2_y}\|\pa_t\phi_{\mu, e_1}-\pa_{x_1}\phi_{\mu, e_1}\|_{L^{\infty}_y}\nonumber\\
&+\|\nabla_y\psi_{\lambda}\|_{L^2_y}\|\nabla_y\phi_{\mu, e_1}\|_{L^{\infty}_y}
\end{align}

Taking $\omega_i=e_1$ in \eqref{wpd}, we find two derivative in $y$ is like one derivatives in $x_1$. By Bernstein inequality, we get:
\begin{align}
	\lvert h_{e_1}\rvert \lesssim& \mu^{\frac{n-1}{4}}\|\pa_t \psi_{\lambda}-\pa_{x_1}\psi_{\lambda}\|_{L^2_y}\|\pa_t\phi_{\mu, e_1}+\pa_{x_1}\phi_{\mu, e_1}\|_{L^{2}_y}\nonumber\\
	&+\mu^{\frac{n-1}{4}}\|\pa_t\psi_{\lambda}+\pa_{x_1}\psi_{\lambda}\|_{L^2_y}\|\pa_t\phi_{\mu, e_1}-\pa_{x_1}\phi_{\mu, e_1}\|_{L^{2}_y}\nonumber\\
	&+\mu^{\frac{n-1}{4}}\|\nabla_y\psi_{\lambda}\|_{L^2_y}\|\nabla_y\phi_{\mu, e_1}\|_{L^{2}_y}.
\end{align}

Combining \eqref{hxg}, we get:
\begin{align}
	I_{e_1}\lesssim \mu^{\frac{n-1}{4}} E^{1/2}(\phi_{\mu,e_1})E^{1/2}(\psi_{\lambda}),
\end{align}

this proves \eqref{fenes}. 

For $Q_{\alpha, \beta}$, we have:
\begin{align}
	\left(\sum\limits_{i,j=1}^{n} \|Q_{i,j}\left(\phi_{\mu},\psi_{\lambda}\right)\|^2_{L^2} +\sum\limits_{i}^{n}\|Q_{0,i}\left(\phi_{\mu},\psi_{\lambda}\right)\|^2_{L^2}\right)^{1/2}
\end{align}
is rotationally equivalent, so we only need to estimate
\begin{align}
	\left(\sum\limits_{i,j}^{n} \|Q_{i,j}\left(\phi_{\mu,e_1},\psi_{\lambda}\right)\|^2_{L^2} +\sum\limits_{i}^{n}\|Q_{0,i}\left(\phi_{\mu,e_1},\psi_{\lambda}\right)\|^2_{L^2}\right)^{1/2},
\end{align}
the rest of proof is completely paralell to the case of $Q_0$.

Thus, the proof of Theorem \ref{wbes} is completed. 
\end{proof}

Then we consider the Theorem \ref{dwbes}.
\begin{proof}
	
If $\mu \ll \lambda$, we have: 
\begin{align*}
\|\left(-\Delta\right)^{-1/2}Q_{i, j}\left(\phi_\mu, \psi_\lambda\right)\|_{L^2([0,1]\times\mathbb{R}^n)}\lesssim \lambda^{-1}\|Q_{i,j}\left(\phi_\mu, \psi_\lambda\right)\|_{L^2([0,1]\times\mathbb{R}^n)}.
\end{align*}

So Theorem \ref{dwbes} follows from Theorem \ref{wbes}. 

We note that $Q_{i,j}(\phi_{\mu}, \psi_\lambda) =\pa_i (\phi_{\mu} \pa_j\psi_{\lambda}) -\pa_j (\phi_{\mu}\pa_i\psi_\lambda)$.  Then by Bernstein inequality and H\"older  inequality,  if $\mu \sim \lambda$, we have:
\begin{align*}
	\|\left(-\Delta\right)^{-1/2}Q_{i,j}\left(\phi_\mu, \psi_\lambda\right)\|_{L^2([0,1]\times\mathbb{R}^n)}\lesssim \lambda \|\phi_{\mu}\|_{L^4([0,1]\times\mathbb{R}^n)} \|\psi_\lambda\|_{L^4([0,1]\times\mathbb{R}^n)}.
\end{align*}

By the following  Strichartz inequality,
\begin{equation}\label{1}
	\|\phi_{\mu}\|_{L^4([0,1]\times\mathbb{R}^n)} \lesssim \mu^{\frac{n-3}{4}-\frac{1}{2}} E^{1/2}\left(\phi_{\mu}\right),
\end{equation}
\begin{equation}\label{2}
\|\psi_{\lambda}\|_{L^4([0,1]\times\mathbb{R}^n)} \lesssim \lambda^{\frac{n-3}{4}-\frac{1}{2}} E^{1/2}\left(\psi_{\lambda}\right),
\end{equation}

we can obtain:
\begin{align}
\|\left(-\Delta\right)^{-1/2}Q_{i,j}\left(\phi_\mu, \psi_\lambda\right)\|_{L^2([0,1]\times\mathbb{R}^n)}&\lesssim \lambda \|\phi_{\mu}\|_{L^4([0,1]\times\mathbb{R}^n)} \|\psi_\lambda\|_{L^4([0,1]\times\mathbb{R}^n)}\nonumber\\
&\lesssim  \mu^{\frac{n-3}{2}}E^{1/2}\left(\phi_{\mu}\right)E^{1/2}\left(\psi_{\lambda}\right)
\end{align}
 for $\lambda \sim \mu$.
Thus, the proof is done.
\end{proof}

Particuly, from the Theorem \ref{wbes} and Theorem \ref{dwbes}, we can derive two corollaries  as follows:
\begin{corollary}\label{th1}
	If we assume that functions $\phi$ and $\psi$ satisfy the \eqref{main1} and \eqref{main2} respectively,  we have:
	\begin{align}
	\| Q\left(\phi, \psi\right)\|_{L^2\left([0,1\right], \dot{B}_{2, 1}^{\frac{n-1}{2}} \left(\mathbb{R}^n\right))} \lesssim\| (\phi_0,\phi_1)\|_{\dot{B}_{2, 1}^{\frac{n+1}{2}}\times \dot{B}_{2, 1}^{\frac{n-1}{2}}} \|(\psi_0, \psi_1)\|_{\dot{B}_{2, 1}^{\frac{n+1}{2}}\times \dot{B}_{2, 1}^{\frac{n-1}{2}}}.
	\end{align}
	for any of null forms $Q_0$  and $Q_{\alpha, \beta}$.   
\end{corollary}

\begin{corollary}\label{th2}
	Let $n=3$. If we assume that functions $\phi$ and $\psi$ satisfy the \eqref{main1} and \eqref{main2} respectively,  we have:
	\begin{align}
	\|\left(-\Delta\right)^{-\frac{1}{2}} Q_{i,j}\left(\phi, \psi\right)\|_{L^2\left(\left[0,1\right] \times \mathbb{R}^3\right)} \lesssim\| (\phi_0,\phi_1)\|_{\dot{B}_{2, 1}^{1}\times \dot{B}_{2, 1}^{0} }\|(\psi_0, \psi_1)\|_{\dot{B}_{2, 1}^{1}\times \dot{B}_{2, 1}^{0}}.
	\end{align}
	for    $1 \leq i , j \leq 3$.
\end{corollary}

Here, we prove the Corollary \ref{th1}, and  Corollary \ref{th2} by Theorem \ref{wbes}, and Theorem \ref{dwbes}.

\begin{proof}
	We first prove Corollary \ref{th1}. From Theorem \ref{wbes}, we have:
	\begin{align*}
		\| Q\left(\phi, \psi\right)\|_{\dot{B}^{\frac{n-1}{2}}_{2,1} \left(\mathbb{R}^n\right)} &= \sum\limits_{\lambda}  \lambda^{\frac{n-1}{2}} \|P_\lambda Q\left(\phi, \psi\right)\|_{L^2 \left(\mathbb{R}^n\right)}\nonumber\\
		&\lesssim \sum\limits_{\lambda}\sum\limits_{\mu <<\lambda } \lambda^{\frac{n-1}{2}} \| Q\left(\phi_\mu, \psi_\lambda\right)\|_{L^2 \left(\mathbb{R}^n\right)} + \sum\limits_{\lambda}\sum\limits_{\mu <<\lambda } \lambda^{\frac{n-1}{2}} \| Q\left(\phi_\lambda,  \psi_{\mu}\right)\|_{L^2 \left(\mathbb{R}^n\right)}\nonumber\\
		&+\sum\limits_{\lambda}\sum\limits_{c^{-1}\lambda \leq \sigma, \sigma\sim\sigma^{'}} \lambda^{\frac{n-1}{2}} \| Q\left(\phi_\sigma,  \psi_{\sigma^{'}}\right)\|_{L^2 \left(\mathbb{R}^n\right)}\nonumber\\
		&\lesssim \sum\limits_{\lambda} \sum\limits_{\mu <<\lambda }\lambda^{\frac{n-1}{2}}\mu^{\frac{n-1}{2}}E^{1/2}\left(\phi_{\mu}\right)E^{1/2}\left(\psi_{\lambda}\right) + \sum\limits_{\lambda} \sum\limits_{\mu <<\lambda }\lambda^{\frac{n-1}{2}}\mu^{\frac{n-1}{2}}E^{1/2}\left(\phi_{\lambda}\right)E^{1/2}\left(\psi_{\mu}\right)  \nonumber\\
		&+\lesssim \sum\limits_{\lambda } \lambda^{\frac{n-1}{2}}\sum\limits_{c^{-1}\lambda \leq \sigma, \sigma\sim\sigma^{'}}\sigma^{'\frac{n-1}{2}}E^{1/2}\left(\phi_{\sigma^{'}}\right)E^{1/2}\left(\psi_{\sigma^{'}}\right).\nonumber\\
        &\lesssim \| (\phi_0,\phi_1)\|_{\dot{B}_{2, 1}^{\frac{n+1}{2}}\times \dot{B}_{2, 1}^{\frac{n-1}{2}}} \|(\psi_0, \psi_1)\|_{\dot{B}_{2, 1}^{\frac{n+1}{2}}\times \dot{B}_{2, 1}^{\frac{n-1}{2}}} 
        \end{align*}

So we obtain:
\begin{equation}
	\| Q\left(\phi, \psi\right)\|_{L^2\left([0,1\right], \dot{B}_{2, 1}^{\frac{n-1}{2}} \left(\mathbb{R}^n\right))} \lesssim \| (\phi_0,\phi_1)\|_{\dot{B}_{2, 1}^{\frac{n+1}{2}}\times \dot{B}_{2, 1}^{\frac{n-1}{2}}} \|(\psi_0, \psi_1)\|_{\dot{B}_{2, 1}^{\frac{n+1}{2}}\times \dot{B}_{2, 1}^{\frac{n-1}{2}}}.
\end{equation}

Then we prove Corollary \ref{th2}. 

From Theorem \ref{dwbes}, we have:
\begin{align*}
	\|\left(-\Delta\right)^{-\frac{1}{2}}Q_{i,j}\left(\phi, \psi\right)\|_{L^2\left(\mathbb{R}^3\right)}&\lesssim \sum\limits_{\lambda}\sum\limits_{ \mu <<\lambda}\lambda^{-1}\|Q_{i,j}\left(\phi_\mu, \psi_\lambda\right)\|_{L^2\left(\mathbb{R}^3\right)} \nonumber\\
	&+\sum\limits_{\lambda}\sum\limits_{\mu \sim \lambda}\|\left(-\Delta\right)^{-\frac{1}{2}}Q_{i,j}\left(\phi_\mu, \psi_\lambda\right)\|_{L^2\left(\mathbb{R}^3\right)}\nonumber\\
	&+\sum\limits_{\lambda}\sum\limits_{ \mu<<\lambda}\lambda
	^{-1}\|Q_{i,j}\left(\phi_\lambda, \psi_\mu\right)\|_{L^2\left(\mathbb{R}^3\right)} \nonumber\\
	&\lesssim \sum\limits_{\lambda}\sum\limits_{ \mu <<\lambda}\lambda^{-1}\mu E^{1/2}\left(\phi_{\mu}\right)E^{1/2}\left(\psi_{\lambda}\right) + \lambda^{-1}\mu E^{1/2}\left(\phi_{\lambda}\right)E^{1/2}\left(\psi_{\mu}\right)\nonumber\\
	&+\sum\limits_{\lambda}\sum\limits_{\mu \sim \lambda} E^{1/2}\left(\phi_{\mu}\right)E^{1/2}\left(\psi_{\lambda}\right)\nonumber\\
	&\lesssim \| (\phi_0,\phi_1)\|_{\dot{B}_{2, 1}^{1}\times \dot{B}_{2, 1}^{0} }\|(\psi_0, \psi_1)\|_{\dot{B}_{2, 1}^{1}\times \dot{B}_{2, 1}^{0}}.
\end{align*}
So we obtain:
\begin{equation}
		\|\left(-\Delta\right)^{-\frac{1}{2}} Q_{i,j}\left(\phi, \psi\right)\|_{L^2\left(\left[0,1\right] \times \mathbb{R}^3\right)} \lesssim \| (\phi_0,\phi_1)\|_{\dot{B}_{2, 1}^{1}\times \dot{B}_{2, 1}^{0} }\|(\psi_0, \psi_1)\|_{\dot{B}_{2, 1}^{1}\times \dot{B}_{2, 1}^{0}}.
\end{equation}
	
Compare to the classical result of Klainerman and Machdon \cite{Klainerman-Machedon}, Corollary \ref{th1} and Corollary \ref{th2}	 have a Besov loss.

\end{proof}

\par{\bf Acknowledgement}:

  The authors would like to thank Professor Sergiu Klainerman for helpful discussions and the anonymous referees for constructive criticism of our manuscript.
  
  This work was supported by the National Natural Science Foundation of China (No. 12171097), the Key Laboratory of Mathematics for Nonlinear Sciences (Fudan University), Ministry of Education of China, P.R.China. Shanghai Key Laboratory for Contemporary Applied Mathematics, School of Mathematical Sciences, Fudan University, P.R. China, and by Shanghai Science and Technology Program [Project No. 21JC1400600].

\end{document}